\documentclass{amsart}
\usepackage{amsmath}
\usepackage{amsthm}
\usepackage{esint}

\usepackage[utf8]{inputenc}
\usepackage[T1]{fontenc}
\usepackage[charter]{mathdesign}
\usepackage{microtype}
\usepackage[bookmarks=true,pdfdisplaydoctitle=true]{hyperref}
\hypersetup{
pdfpagemode=UseNone,
pdfstartview=FitH,
pdftitle={Weighted and vector-valued variational estimates for ergodic averages},
pdfauthor={Ben Krause and Pavel Zorin-Kranich},
pdfsubject={Mathematics Subject Classification (2010): 37A45 (Primary), 26A45 (Secondary)},
pdfkeywords={variation norm, ergodic averages},
pdflang=en-US
}
\theoremstyle{plain}

\newtheorem{theorem}[equation]{Theorem}

\newtheorem{proposition}[equation]{Proposition}
\newtheorem{lemma}[equation]{Lemma}

\newtheorem{corollary}[equation]{Corollary}

\usepackage[safeinputenc=true,style=alphabetic,firstinits,maxnames=4,doi=false,eprint=false,isbn=false,url=false,backend=bibtex]{biblatex}

\newbibmacro{string+doiurl}[1]{%
  \iffieldundef{doi}{%
    \iffieldundef{url}{%
      \iffieldundef{eprint}{%
        #1%
      }{
        \href{http://arxiv.org/abs/\thefield{eprint}}{#1}
        }%
      }{%
      \href{\thefield{url}}{#1}%
    }%
  }{%
    \href{http://dx.doi.org/\thefield{doi}}{#1}%
  }%
}
\DeclareFieldFormat[article,misc]{title}{\usebibmacro{string+doiurl}{\mkbibquote{#1}}}

\newcommand*{\mailto}[1]{\href{mailto:#1}{#1}}

\newcommand*{\Z}{\mathbb{Z}}

\newcommand*{\R}{\mathbb{R}}

\def\<{\left\langle}
\def\>{\right\rangle}
\newcommand{\dif}{\mathrm{d}}
\newcommand{\BMO}{\mathrm{BMO}}
\newcommand{\diam}{\ell}
\newcommand{\dist}{\mathrm{dist}}
\newcommand*{\tdif}[3][]{\frac{\dif^{ #1} #2}{\dif { #3}^{ #1}}}
\newcommand{\V}[3][r]{\| #2 \|_{V^{#1}_{#3}}}
\newcommand{\hV}[3][r]{\| #2 \|_{\tilde V^{#1}_{#3}}}
\newcommand{\s}{\tilde S}
\renewcommand{\r}{\tilde R}
\numberwithin{equation}{section}

\begin{document}
\title{Weighted and vector-valued variational estimates for ergodic averages}
\author{Ben Krause}
\address{UCLA Math Sciences Building\\
  Los Angeles\\
  CA 90095-1555\\
  USA}
\email{\mailto{benkrause23@math.ucla.edu}}
\author{Pavel Zorin-Kranich}
\address{Universität Bonn\\
  Mathematisches Institut\\
  Endenicher Allee 60\\
  53115 Bonn\\
  Germany
}
\email{\mailto{pzorin@uni-bonn.de}}
\urladdr{\url{http://www.math.uni-bonn.de/people/pzorin/}}
\begin{abstract}
We prove weighted and vector-valued variational estimates for ergodic averages on $\R^d$.
The weighted square function estimate relating ergodic averages to the dyadic martingale is obtained using an $\ell^r$ version of a reverse Hölder inequality for variation seminorms.
\end{abstract}
\maketitle

\section{Introduction}
We denote ergodic averages on $\R^{d}$, $d\geq 1$, by
\[
A_{t}f(x) = \fint_{B(0,t)} f(x+y) \dif y,
\]
where $B(0,t)$ is the ball of radius $t$ centered at zero and $\fint_{A} = |A|^{-1}\int_{A}$.
The homogeneous $r$-variation norm is denoted by
\[
\hV{a_{t}}{t\in I} := \sup_{t_{0}<t_{1}<\dots<t_{J} \in I} \big( \sum_{j=1}^{J} |a_{t_{j}} - a_{t_{j-1}}|^{r} \big)^{1/r}
\]
and the inhomogeneous $r$-variation norm by
\[
\V{a_{t}}{t\in I} := \sup_{t\in I} |a_{t}| + \hV{a_{t}}{t\in I}.
\]
Variational estimates for ergodic averages of the form
\[
\| \V{A_{t}f(x)}{t>0} \|_{L^{p}_{x}} \lesssim \|f\|_{L^{p}}
\]
have been introduced by Bourgain \cite{MR1019960} and a complete theory has been developed by a number of authors \cite{MR1996394,MR1645330,MR2434308} covering the full possible range of exponents $p,r$ including the end points.
Our purpose is to obtain weighted and vector-valued versions of these results.

A \emph{weight} is a positive function on $\R^{d}$.
The $A_{1}$ constant of a weight is given by
\begin{equation}
\label{eq:A1}
[w]_{A_{1}} = \sup_{x} Mw(x)/w(x),
\end{equation}
the $A_{p}$ constant, $1<p<\infty$, is given by
\begin{equation}
\label{eq:Ap}
[w]_{A_{p}} = \sup_{Q} \big( |Q|^{-1} \int_{Q} w \big) \big( |Q|^{-1} \int_{Q} w^{-1/(p-1)} \big)^{p-1},
\end{equation}
and a weight $w$ is $A_{\infty}$ if
\begin{equation}
\label{eq:Ainfty}
w(E)/w(Q) \leq C_{w} (|E|/|Q|)^{\delta}
\end{equation}
for some $\delta\in(0,1]$ and all $E\subset Q$.
Here and later the letter $Q$ denotes a cube in $\R^{d}$, not necessarily dyadic unless explicitly stated so.
We refer to \cite[Chapter IV]{MR807149} for the basic properties of $A_{p}$ weights.
Note that all constants in the results from \cite[Chapter IV]{MR807149} that we use can be taken to depend only on the $A_{p}$ constants of the weights involved in them and not otherwise on the weight.
This kind of dependence is required for the Rubio de Francia extrapolation theorem to apply (see e.g.\ \cite[Theorem 3.1]{MR2754896}).

The set of dyadic cubes in $\R^{d}$ with side length $2^{k}$ is denoted by $\mathcal{Q}_{k}$.
For a point $x$ we write $\mathcal{Q}_{k}(x)$ for the unique (away from a set of measure zero) dyadic cube with side length $2^{k}$ containing $x$.
The conditional expectation onto the $\sigma$-algebra generated by $\mathcal{Q}_{k}$ is denoted by $E_{k}$.

Ergodic averages will be compared with a dyadic martingale using the short variations
\[
S_{k}f := \V[2]{A_{t}f-E_{k}f}{2^{k}\leq t\leq 2^{k+1}}.
\]
Following \cite{2013arXiv1309.2336K} we will in fact use the (larger) smoothed version of the short variations given by
\[
\s_{k}f(x) := \sup_{y\in 3\mathcal{Q}_{k}(x)} S_{k} f(y).
\]
Here, $3\mathcal{Q}_k(x)$ denotes the concentric cube to $\mathcal{Q}_k(x)$ with three times the side length.
The smoothed square function is defined by
\[
\s f := (\sum_{k} (\s_{k} f)^{2})^{1/2}.
\]
Weighted bounds for a discrete version of $\s$ have been proved in \cite{2013arXiv1309.2336K} using the Auscher--Martell extrapolation theorem with limited range of exponents (see \cite[Theorem 7.1]{MR2754896}).
We present a direct proof relying on an $\ell^{r}$ version of a reverse H\"older inequality from \cite{MR1996394} and another alternative proof using a good $\lambda$ inequality.
\begin{theorem}
\label{thm:square}
The operator $\s$ is bounded
\begin{enumerate}
\item\label{thm:square:Ap} on $L^{p}(w)$ for every $w\in A_{p}$, $1<p<\infty$, and
\item\label{thm:square:A1} as an operator $L^{1}(w)\to L^{1,\infty}(w)$ for every $w\in A_{1}$.
\end{enumerate}
In each case the operator norm is bounded in terms of the $A_{p}$ constant of the weight.
\end{theorem}
For an $\ell^{\infty}\to\BMO$ estimate in the discrete setting see also \cite[Section 3]{2013arXiv1309.2336K}.
The above estimate for the square function together with weighted bounds for the dyadic martingale square function and Haar multipliers will be used to show weighted and vector-valued bounds for the $r$-variation of ergodic averages.
\begin{theorem}
\label{thm:var}
Let $r>2$ and $1<q<\infty$.
The operator
\[
f_{i}(x) \mapsto (A_{t}f_{i})(x)
\]
acting on sequence-valued functions on $\R^{d}$ is bounded as an operator
\begin{align}
L^{1}_{x}(\ell^{q}_{i},w) &\to L^{1,\infty}_{x}(\ell^{q}_{i}(V^{r}_{t}), w), \quad\text{for any } w\in A_{1}(\R^{d}) \label{eq:thm:1}\\
L^{\infty}_{x}(\ell^{q}_{i}) &\to \BMO_{x}(\ell^{q}_{i}(V^{r}_{t})), \quad\text{and} \label{eq:thm:bmo}\\
L^{p}_{x}(\ell^{q}_{i},w) &\to L^{p}_{x}(\ell^{q}_{i}(V^{r}_{t}),w) \quad\text{for any } 1<p<\infty \text{ and } w\in A_{p}(\R^{d}). \label{eq:thm:p}
\end{align}
In each case the operator norm is bounded in terms of the $A_{p}$ constant of the weight and grows as $\frac{r}{r-2}$ for $r\to 2$.
\end{theorem}
Theorem~\ref{thm:var} is a joint generalization of the Fefferman--Stein vector-valued maximal inequality \cite{MR0284802} and the Bourgain--Lépingle inequality \cite{MR1019960}.
Since bounds for a fixed $t$ are immediate, it suffices to treat the homogeneous variation.

Our proof also gives analogous results for dilates of other neighborhoods of the origin provided that the boundary is sufficiently smooth, this applies e.g.\ to cubes.
Throughout the proof we assume $r<\infty$, which is justified by the fact that the conclusion of the theorem becomes stronger as $r$ decreases.
We also assume that the coordinate functions $f_{i}$ are smooth on $\R^{d}$, the general case follows by standard truncation arguments.

The $p$-maximal function and the sharp $p$-maximal function are denoted by
\begin{equation}
\label{eq:max-fct}
M_{p}f(x) = \sup_{Q\ni x} \big( |Q|^{-1} \int_{Q} |f|^{p} \big)^{1/p},
\quad
M^{\sharp}_{p}f(x) = \sup_{Q\ni x} \inf_{c} \big( |Q|^{-1} \int_{Q} |f-c|^{p} \big)^{1/p}.
\end{equation}
The subscript $p$ is omitted if $p=1$.

The authors thank the Hausdorff Research Institute for Mathematics for hospitality during the Trimester Program ``Harmonic Analysis and Partial Differential Equations''.

\section{A reverse Hölder inequality}
For $1<r<\infty$ let
\[
R_{k}b(x) := \V{A_{t}b}{t\in [2^{k},2^{k+1}]},
\quad
\r_{k}b(x) = \sup_{y\in 3Q_{k}(x)} R_{k}b(y),
\quad
\r (b) := (\sum_{k} \r_{k}(b)^{r})^{1/r}.
\]
Denote the side length of a cube $Q$ by $\diam(Q)$.
\begin{lemma}[{cf.\ \cite[Lemma 4.2]{MR1996394}}]
\label{lem:reverse-holder-single-scale}
Let $1<r<\infty$.
Let $\mathcal{Q}$ be a collection of disjoint cubes $Q$ of size $\lesssim 2^{k}$.
For each $Q$ let $b^{Q}$ be a scalar-valued function supported on $Q$ with $\int b^{Q} = 0$.
Then for every $\alpha> \frac{d-1}{r'}$ we have
\begin{equation}
\label{eq:V-bi-biQ-single-scale}
R_{k}(\sum_{Q} b^{Q}(x))^{r}
\lesssim_{\alpha,r,d}
(2^{k})^{\alpha r} \sum_{Q} \diam(Q)^{-\alpha r} R_{k}(b^{Q}(x))^{r}.
\end{equation}
\end{lemma}
\begin{proof}
We consider only the homogeneous variation, in order to get the inhomogeneous variation it suffices to additionally consider an arbitrary (but fixed) $t$, which is similar but easier.
Fix an arbitrary sequence $2^{k}\leq t_{1}<\dots<t_{J} \leq 2^{k+1}$ and split the sum over cubes in
\[
\sum_{j} |\sum_{Q} (A_{t_{j+1}}-A_{t_{j}})b^{Q}(x)|^{r}.
\]
according to the size of the cubes
\[
\sum_{j} |\sum_{i\leq 0} \sum_{Q : \diam(Q) \approx 2^{k+i}} (A_{t_{j+1}}-A_{t_{j}})b^{Q}(x)|^{r}.
\]
Let $\alpha$ be chosen later and apply Hölder's inequality in the sum over $i$ to estimate this by
\begin{multline}
\label{eq:V-bi-biQ:CS-single-scale}
\sum_{j}
\big( \sum_{i\leq 0} \sum_{Q : \diam(Q) \approx 2^{k+i}} |2^{-\alpha i}(A_{t_{j+1}}-A_{t_{j}})b^{Q}(x)|^{r} \big)\\
\cdot \big(\sum_{i\leq 0} \sum_{Q : \diam(Q) \approx 2^{k+i}} |2^{\alpha i} 1_{Q:Q\cap (\partial B(x,t))\neq\emptyset, t=t_{j} \text{ or }t_{j+1}}|^{r'} \big)^{r/r'}.
\end{multline}
In order to estimate the second bracket note that $Q\cap (\partial B(x,t))\neq\emptyset$ implies $Q$ is contained in a spherical shell of radius $\approx 2^{k}$ and thickness $\approx 2^{k+i}$.
For a fixed $i$ there can be at most $O(2^{-i(d-1)})$ cubes of this kind.
Hence the second bracket in \eqref{eq:V-bi-biQ:CS-single-scale} is bounded by
\[
\big(\sum_{i\leq 0} 2^{\alpha i r'} 2^{-i(d-1)} \big)^{r/r'},
\]
which is finite by the assumption on $\alpha$.
This gives for \eqref{eq:V-bi-biQ:CS-single-scale} the bound
\begin{multline*}
\sum_{j}
\big( \sum_{i\leq 0} \sum_{Q : \diam(Q) \approx 2^{k+i}} |2^{-\alpha i}(A_{t_{j+1}}-A_{t_{j}})b^{Q}(x)|^{r} \big)\\
\lesssim
\sum_{j} \sum_{Q} \diam(Q)^{-\alpha r} (2^{k})^{\alpha r} |(A_{t_{j+1}}-A_{t_{j}})b^{Q}(x)|^{r}
\end{multline*}
Taking the supremum over sequences $(t_{j})_{j}$ we obtain \eqref{eq:V-bi-biQ-single-scale}.
\end{proof}
Estimating the $r$-variation norm by the $1$-variation norm and noting that only cubes with $\dist(x,Q) \lesssim 2^{k}$ contribute to the sum we obtain
\begin{corollary}
\label{cor:reverse-holder-single-scale}
In the situation of Lemma~\ref{lem:reverse-holder-single-scale} we have
\[
\r_{k}(\sum_{Q} b^{Q}(x))^{r}
\lesssim_{\alpha,r,d}
(2^{k})^{\alpha r-d r} \sum_{Q : \dist(x,Q) \lesssim 2^{k}} \diam(Q)^{-\alpha r} \|b^{Q}\|_{1}^{r}.
\]
\end{corollary}

\section{Strong type bounds for the square function}
Recall that a Haar function on a dyadic cube is a function that is constant on each dyadic subcube and has integral zero.
A Haar function is $L^{\infty}$ normalized if it is bounded by $1$.
\begin{lemma}
\label{lem:short-variation-single-scale}
Let $k\in\Z$ and $j<0$.
Let $h_{Q}$ be $L^{\infty}$ normalized Haar functions supported on the cubes $Q\in\mathcal{Q}_{k+j+1}$.
Let $1<p<\infty$ and $w\in A_{p}$.
Then
\begin{equation}
\label{eq:short-variation-single-scale:claim}
\|\s_{k}(\sum_{Q\in\mathcal{Q}_{k+j+1}} h_{Q} E_{k+j+1} f)\|_{L^{p}(w)}
\lesssim
2^{-\epsilon |j|} \|f\|_{L^{p}(w)},
\end{equation}
where $\epsilon>0$ and the implied constant depend only on $p$ and $[w]_{A_{p}}$, but not on $k,j$, or the Haar functions.
\end{lemma}
\begin{proof}
By homogeneity we may assume $k=0$.
Let $w\in A_{p}$, then there exists $r\in (1,\min(2,p))$ such that $w\in A_{p/r}$.
Write $f_{Q}$ for the value of $E_{j+1}f$ on $Q$ and observe $E_{0} \sum_{Q\in\mathcal{Q}_{j+1}} h_{Q} f_{Q}=0$.
Hence we have
\[
\|\s_{0}(\sum_{Q\in\mathcal{Q}_{j+1}} h_{Q} E_{j+1} f)\|_{L^{p}(w)}
\lesssim
\|\r_{0}(\sum_{Q\in\mathcal{Q}_{j+1}} h_{Q} f_{Q})\|_{L^{p}(w)},
\]
where $\r_{0}$ is defined using the above value of $r$.
Applying Corollary~\ref{cor:reverse-holder-single-scale} we obtain
\[
\r_{0}(\sum_{Q\in\mathcal{Q}_{j+1}} h_{Q} f_{Q})^{r}(x)
\lesssim
2^{-r\alpha j} \sum_{Q\in \mathcal{Q}_{j+1}, \dist(x,Q)\lesssim 1} \|h_{Q}f_{Q}\|_{1}^{r}.
\]
By H\"older's inequality this is bounded by
\begin{align*}
2^{-r\alpha j} \sum_{Q\in \mathcal{Q}_{j+1}, \dist(x,Q)\lesssim 1} \|1_{Q}f_{Q}\|_{r}^{r} \|h_{Q}\|_{r'}^{r}
&\leq
2^{-r\alpha j} 2^{jd r/r'}\int_{B(x,C)} |E_{j+1}f|^{r}\\
&\lesssim
2^{r(d/r'-\alpha) j} M_{r}(E_{j+1}f)^{r}(x).
\end{align*}
Thus we have established
\[
\r_{0}(\sum_{Q\in\mathcal{Q}_{j+1}} h_{Q} f_{Q})(x)
\lesssim
2^{\epsilon j} M_{r}(E_{j+1}f)(x)
\]
for any $\epsilon<1/r'$.
We are done since the maximal function $M_{r}$ is bounded on $L^{p}(w)$.
\end{proof}

\begin{proof}[Proof of Theorem~\ref{thm:square}, part~\ref{thm:square:Ap}]
By Rubio de Francia's extrapolation theorem it suffices to consider $p=2$.
Write $f=\sum_{j} d_{j}$, $d_{j}=E_{j}f-E_{j+1}f$.
The result will follow if we can show
\[
\|\s_{k}(d_{k+j})\|_{L^{2}(w)} \lesssim 2^{-\epsilon |j|} \|d_{k+j}\|_{L^{2}(w)}.
\]
Indeed,
\[
\s f
=
(\sum_{k} (\s_{k}(\sum_{j}d_{k+j}))^{2})^{1/2}
\leq
\sum_{j} (\sum_{k} (\s_{k}(d_{k+j}))^{2})^{1/2},
\]
and taking $L^{2}(w)$ norms on both sides we obtain
\[
\| \s f \|_{L^{2}(w)}
\leq
\sum_{j} (\sum_{k} \|\s_{k}(d_{k+j}) \|_{L^{2}(w)}^{2})^{1/2}
\lesssim
\sum_{j} 2^{-\epsilon |j|} (\sum_{k} \|d_{k+j}\|_{L^{2}(w)}^{2})^{1/2}.
\]
The sum over $k$ on the right-hand side is the dyadic martingale square function which is bounded on $L^{2}(w)$, see \cite[Theorem 3.6]{MR1214060} or \cite{MR2657437}.

For $j\geq -10^{d}$ (say) we note that $S_k$ is invariant under constant addition to use the estimate
\[
S_{k}f(x)
\lesssim
\inf_{c} 2^{-kd} \int_{B_{C2^{k}}(x)} |f-c|.
\]
It follows that an estimate of the same kind holds for $\s_{k}$.
Writing $d_{k+j}=\sum_{Q\in\mathcal{Q}_{k+j}} a_{Q} 1_{Q}$ we obtain
\[
\s_{k}d_{k+j} \lesssim \sum_{Q\in\mathcal{Q}_{k+j}} |a_{Q}| 1_{\partial Q + B_{C 2^{k}}}.
\]
The overlap of the characteristic functions on the right-hand side is bounded by an absolute constant, so
\[
(\s_{k}d_{k+j})^{2} \lesssim \sum_{Q\in\mathcal{Q}_{k+j}} |a_{Q}|^{2} 1_{\partial Q + B_{C 2^{k}}}.
\]
Integrating this we obtain
\[
\| \s_{k}d_{k+j} \|_{L^{2}(w)}^{2}
\lesssim
\sum_{Q\in\mathcal{Q}_{k+j}} |a_{Q}|^{2} w(\partial Q + B_{C 2^{k}})
\]
Since the $A_{2}$ weight $w$ also satisfies the $A_{\infty}$ condition, this is bounded by
\[
\sum_{Q\in\mathcal{Q}_{k+j}} |a_{Q}|^{2} w(Q + Q_{C 2^{k}}) (|\partial Q + Q_{C 2^{k}}| / |Q + Q_{C 2^{k}}|)^{\delta}
\]
for some $\delta>0$, and since $w$ is doubling this is bounded by
\[
\sum_{Q\in\mathcal{Q}_{k+j}} |a_{Q}|^{2} w(Q) (2^{-j})^{\delta}
=
2^{-j\delta} \| d_{k+j} \|_{L^{2}(w)}^{2}.
\]

For $j<-10^{d}$ write
\[
d_{k+j} = \sum_{Q\in\mathcal{Q}_{k+j+1}} h_{Q} E_{k+j+1}f,
\]
where $h_{Q} = d_{k+j}1_{Q}/\|d_{k+j}1_{Q}\|_{\infty}$ are normalized Haar functions (with convention $h_{Q}=0$ if $d_{k+j}1_{Q}=0$)
and $f=\sum_{Q\in\mathcal{Q}_{k+j+1}} 1_{Q} \|d_{k+j}1_{Q}\|_{\infty}$.
Note $\|f\|_{L^{2}(w)} \lesssim \|d_{k+j}\|_{L^{2}(w)}$ since $w$ is doubling and $d_{k+j}$ is constant at scale $2^{k+j}$.
The claim now follows from Lemma~\ref{lem:short-variation-single-scale}.
\end{proof}

\section{The jump inequality}
Let $(a_{t})_{t}$ be an arbitrary function.
The \emph{jump counting function} $N_{\lambda}(a_{t})$ is the supremum over all $J$ such that there exist $t_{0}<t_{1}<\dots<t_{J}$ with $|a_{t_{j}}-a_{t_{j-1}}|>\lambda$ for all $j=1,\dots,J$.
See \cite{MR2434308} for the basic properties of the jump counting function and its relation to the variational norms.
\begin{proposition}
\label{prop:jump}
Let $w\in A_{p}(\R^{d})$, $1<p<\infty$.
Then
\[
\sup_{\lambda>0} \lambda^{-1} \| \sqrt{N_{\lambda}(A_{t}f)} \|_{L^{p}(w)} \lesssim \| f \|_{L^{p}(w)}.
\]
\end{proposition}
Proposition~\ref{prop:jump} implies the scalar-valued case of \eqref{eq:thm:p} in Theorem~\ref{thm:var} by the interpolation argument in \cite[Section 2]{MR2434308} (note that an $A_{p}$ weight, $1<p<\infty$, is also in $A_{q}$ for all $q$ in a neighborhood of $p$).
Extrapolation then yields the vector-valued case of \eqref{eq:thm:p}.
\begin{proof}
In view of the weighted bound for the square function (part~\ref{thm:square:Ap} of Theorem~\ref{thm:square}) it suffices to show
\[
\sup_{\lambda>0} \lambda^{-1} \| \sqrt{N_{\lambda}(E_{t}f)} \|_{L^{p}(w)} \lesssim \| f \|_{L^{p}(w)},
\]
where $t$ takes dyadic values (here we use the convention that $E_{k}$ is the conditional expectation at scale $2^{-k}$), see e.g.\ \cite[Lemma 1.3]{MR2434308}.
By the greedy selection argument, see e.g.\ \cite[Lemma 6.7]{MR1645330}, this follows from
\[
\| (\sum_{j}|E_{t_{j+1}(x)}f(x)-E_{t_{j}(x)}f(x)|^{2} )^{1/2} \|_{L^{p}_{x}(w)} \lesssim \| f \|_{L^{p}(w)},
\]
where $t_{1}\leq t_{2}\leq\dots$ are stopping times and the bound does not depend on the stopping times.
By truncation we may assume $t_{0} = -\infty$ and $t_{J}=+\infty$ for some $J$.
Writing $r_{j}$ for the Rademacher functions the left-hand side can be estimated by
\[
\| \| \sum_{j} r_{j}(s)(E_{t_{j+1}(x)}f(x)-E_{t_{j}(x)}f(x)) \|_{L^{p}_{s}} \|_{L^{p}_{x}(w)}
=
\| \| \sum_{j} r_{j}(s) \sum_{k=t_{j}(x)}^{t_{j+1}(x)-1} d_{k}(x) \|_{L^{p}_{x}(w)} \|_{L^{p}_{s}},
\]
where $d_{k} = E_{k+1}f-E_{k}f$.
Now the $L^{p}_{x}(w)$ norm on the right-hand side can be written as
\[
\| \sum_{k} r_{\max \{j : t_{j}(x) \leq k\}}(s) d_{k}(x) \|_{L^{p}_{x}(w)}.
\]
The function $x\mapsto r_{\max \{j : t_{j}(x) \leq k\}}(s)$ is constant at scale $2^{-k}$ since $t_{j}$ are stopping times, so this is bounded by $\|f\|_{L^{p}(w)}$ uniformly in $s$ in view of the weighted bound for the Haar multipliers (which extends to all $1<p<\infty$ by extrapolation), see \cite[Theorem 3.1]{MR1748283} for the case $d=1$ and \cite[Theorem 1.6]{MR2657437} for the general case.
Integrating over $s$ we obtain the claim.
\end{proof}

\section{Weak type \texorpdfstring{$(1,1)$}{(1,1)} bounds}
\begin{proposition}
\label{prop:weak11}
Let $1<q<\infty$, let $w$ be an $A^{1}$ weight and let $T$ be a sublinear operator that is bounded on $L^{q}(w)$ and pointwise bounded by a finite linear combination of operators of the form $\r$ for some $1<r\leq q$.
Then $T : L^{1}(\ell^{q},w) \to L^{1,\infty}(\ell^{q},w)$.
\end{proposition}
This allows us to deduce part~\ref{thm:square:A1} of Theorem~\ref{thm:square} from part~\ref{thm:square:Ap} of that theorem and part~\eqref{eq:thm:1} of Theorem~\ref{thm:var} from part~\eqref{eq:thm:p} of that theorem.
\begin{proof}
The argument is adapted from \cite{MR1645330,MR1996394}.
By homogeneity it suffices to prove
\[
w\{ x : (\sum_{i} |Tf_{i}|^{q})^{1/q} > 1\}
\lesssim
\| (\sum_{i} |f_{i}|^{q})^{1/q} \|_{L^{1}(w)}.
\]
We use the Calderón--Zygmund decomposition as in \cite{MR0284802}.
Let $F = (\sum_{i} |f_{i}|^{q})^{1/q}$, then there exist disjoint cubes (denoted by $Q$) such that $\|F\|_{L^{\infty}(\R \setminus \cup Q)} \leq 1$, $\sum_{Q} |Q| \lesssim \|F\|_{1}$, and $\|F\|_{L^{1}(Q)} \lesssim |Q|$.
Let
\[
g_{j}(x) =
\begin{cases}
|Q|^{-1} \int_{Q} f_{j}, & x\in Q,\\
f_{j}(x), & x\not\in\cup Q
\end{cases}
\]
and
\[
b_{j} = \sum_{Q} b_{j}^{Q},
\quad
b_{j}^{Q}(x) =
\begin{cases}
f_{j}(x) - |Q|^{-1} \int_{Q} f_{j}, & x\in Q,\\
0, & x\not\in\cup Q
\end{cases}
\]
Let $G = (\sum_{j} |g_{j}|^{q})^{1/q}$, then $G(x)=F(x)$ for $x\not\in\cup Q$ and
\[
G(x)
=
(\sum_{j} ||Q|^{-1} \int_{Q} f_{j}|^{q})^{1/q}
\leq
|Q|^{-1} \int_{Q} (\sum_{j} |f_{j}|^{q})^{1/q}
=
|Q|^{-1} \| F \|_{L^{1}(Q)}
\lesssim 1
\]
for $x\in Q$.
Hence $\|G\|_{L^{1}(w)} \leq \|F\|_{L^{1}(w)}$ and $\|G\|_{\infty} \lesssim 1$, so we get the required weak type bound for $(g_{j})_{j}$ from the strong type $(q,q)$ bound.

Note
\[
w(5Q)
\leq
\int_{Q} |F| w(5Q)/|Q|
\lesssim
\int_{Q} |F| Mw
\lesssim
\int_{Q} |F| w,
\]
so with $\tilde E := \cup_{Q} 5Q$ it suffices to show
\begin{equation}
\label{eq:outside-exceptional}
w\{ x\not\in \tilde E : (\sum_{i} |Tb_{i}|^{q})^{1/q} > 1 \}
\lesssim
\sum_{Q} w(Q).
\end{equation}
By the hypothesis it suffices to show
\begin{equation}
\label{eq:outside-exceptional-Lr}
\int_{\R^{d} \setminus \tilde E} (\sum_{i} (\sum_{k} (\r_{k}b_{i})^{r})^{q/r})^{r/q} w
\lesssim
\sum_{Q} w(Q).
\end{equation}

Let $\alpha>(d-1)/r'$ be chosen later.
For $x\not\in E$ only cubes with $\diam(Q) \lesssim 2^{k}$ contribute to $\r_{k}$.
Thus by Corollary~\ref{cor:reverse-holder-single-scale} we have
\[
(\r_{k} b_{i})^{r}(x)
\lesssim
2^{k\alpha r-kdr} \sum_{Q : \dist(x,Q)\lesssim 2^{k}} \diam(Q)^{-\alpha r} \|b_{i}^{Q}\|_{1}^{r}.
\]
Hence the left-hand side of \eqref{eq:outside-exceptional-Lr} can be estimated by
\begin{equation}
\label{eq:outside-exceptional-rev-holder}
\int_{\R^{d} \setminus \tilde E} (\sum_{i} (\sum_{k} 2^{k(\alpha-d) r} \sum_{Q} 1_{\dist(\cdot,Q)\lesssim 2^{k}}(x) \diam(Q)^{-\alpha r} \|b_{i}^{Q}\|_{1}^{r})^{q/r})^{r/q} w(x) \dif x.
\end{equation}
By Minkowski's inequality this is bounded by
\[
\sum_{Q} \diam(Q)^{-\alpha r} (\sum_{i} \|b_{i}^{Q}\|_{1}^{q})^{r/q} \sum_{k} 2^{k(\alpha-d) r} \int_{\R^{d} \setminus \tilde E}  1_{\dist(\cdot,Q)\lesssim 2^{k}}(x) w(x) \dif x.
\]
Note that
\begin{equation}
\label{eq:sum-biQ-L1}
\sum_{i} \|b_{i}^{Q}\|_{1}^{q}
\leq
\big(\int_{Q} (\sum_{i} |b_{Q}^{i}|^{q})^{1/q} \big)^{q}
\lesssim
\|F\|_{L^{1}(Q)}^{q}
\lesssim
|Q|^{q}
\end{equation}
by Minkowski's inequality.
Hence we obtain the bound
\[
\sum_{Q} \diam(Q)^{-\alpha r} |Q|^{r} \sum_{k : 2^{k} \gtrsim \diam(Q)} 2^{k(\alpha-d) r} \int_{\R^{d}}  1_{\dist(\cdot,Q)\lesssim 2^{k}}(x) w(x) \dif x.
\]
The latter integral can be estimated by $2^{kd} \min_{x\in Q}Mw(x)$, and the sum over $k$ is finite provided $(\alpha-d)r+d<0$.
In this case we obtain the bound
\begin{multline*}
\sum_{Q} \diam(Q)^{-\alpha r} |Q|^{r} \diam(Q)^{(\alpha-d) r+d} \min_{x\in Q}Mw(x)\\
\lesssim
\sum_{Q} |Q|^{r} \diam(Q)^{d(1-r)} \min_{x\in Q}w(x)\\
\approx
\sum_{Q} |Q| \min_{x\in Q}w(x)\\
\leq
\sum_{Q} w(Q).
\end{multline*}
as required.
Note that the restrictions on $\alpha$ imposed here are equivalent to $\frac{d-1}{r'} < \alpha < \frac{d}{r'}$, so an appropriate $\alpha$ can be chosen whenever $r>1$.
\end{proof}

The above proof also yields an $A_{1}$ weighted weak type $(1,1)$ estimate for the jump counting function.

\section{The BMO bound}
Finally we prove \eqref{eq:thm:bmo} in Theorem~\ref{thm:var}.
Suppose $\| (\sum_{i} |f_{i}|^{q})^{1/q} \|_{\infty}\leq 1$ and let $Q$ be a cube.
It suffices to show
\[
\fint_{Q} (\sum_{i} \V{A_{t}f_{i}(x) - c_{t,i,Q}}{t}^{q})^{1/q} \dif x \lesssim 1
\]
for some choice of functions $c_{t,i,Q}$.
We split $f=g+b$, where $g=f 1_{3Q}$.
For the local part $g$ we have $\|g\|_{L^{q}(\ell^{q})} \lesssim |Q|^{1/q}$, and it follows from the scalar-valued case of the Bourgain--Lépingle inequality \ref{eq:thm:p} that
\[
\| (\sum_{i} \V{A_{t}g_{i}(x)}{t}^{q})^{1/q} \|_{L^{q}_{x}} \lesssim \|g\|_{L^{q}(\ell^{q})}.
\]
The required estimate for the local part then follows by an application of the Cauchy--Schwarz inequality in the $x$ variable.

It remains to treat the global part $b$.
We set $c_{t,i,Q}=\fint_{Q}A_{t}b_{i}(y)\dif y$.
We have
\[
\fint_{Q} (\sum_{i} \V{A_{t}b_{i}(x) - c_{t,i,Q}}{t}^{q})^{1/q} \dif x
\leq
\fint_{Q}\fint_{Q} (\sum_{i} \V{A_{t}b_{i}(x) - A_{t}b_{i}(y)}{t}^{q})^{1/q} \dif x \dif y,
\]
so it suffices to show
\begin{equation}
\label{eq:BMO-local}
(\sum_{i} \V{A_{t}b_{i}(x) - A_{t}b_{i}(y)}{t}^{q})^{1/q} \lesssim 1
\end{equation}
uniformly in $x,y\in Q$.
Estimating the $r$-variational norm by the $\min(r,q)$-variational norm we may assume $r\leq q$ for the remaining part of the proof.
By Lemma~\ref{lem:besov-vr} we have
\[
\V{A_{t}b_{i}(x) - A_{t}b_{i}(y)}{t}
\lesssim
\|A_{t}b_{i}(x) - A_{t}b_{i}(y)\|_{L^{r}_{t}}^{1-1/r}
\|\tdif{}{t}(A_{t}b_{i}(x) - A_{t}b_{i}(y))\|_{L^{r}_{t}}^{1/r}.
\]
By Cauchy--Schwarz the left-hand side of \eqref{eq:BMO-local} can be estimated by
\[
(\sum_{i} \|A_{t}b_{i}(x) - A_{t}b_{i}(y)\|_{L^{r}_{t}}^{q})^{(1-1/r)/q}
(\sum_{i} \|\tdif{}{t}(A_{t}b_{i}(x) - A_{t}b_{i}(y))\|_{L^{r}_{t}}^{q})^{(1/r)/q}.
\]
By Minkowski's integral inequality this is bounded by
\begin{equation}
\label{eq:BMO-local2}
\|(\sum_{i} |A_{t}b_{i}(x) - A_{t}b_{i}(y)|^{q})^{1/q}\|_{L^{r}_{t}}^{1-1/r}
\|(\sum_{i} |\tdif{}{t}(A_{t}b_{i}(x) - A_{t}b_{i}(y))|^{q})^{1/q}\|_{L^{r}_{t}}^{1/r}.
\end{equation}
Note that $A_{t}b_{i}(x) = A_{t}b_{i}(y) = 0$ if $t<\diam(Q)$.
For $t>\diam(Q)$ we write
\[
A_{t}b_{i}(x) - A_{t}b_{i}(y)
=
\frac{C}{t^{d}} \int (1_{B(x,t)}-1_{B(y,t)}) b_{i}
\]
and use the fact that
\begin{equation}
\label{eq:smoothness}
|B(x,t) \Delta B(y,t)| \lesssim |x-y| t^{d-1} \lesssim \diam(Q) t^{d-1}.
\end{equation}
This gives the estimate
\begin{multline*}
\|(\sum_{i} |A_{t}b_{i}(x) - A_{t}b_{i}(y)|^{q})^{1/q}\|_{L^{r}_{t>\diam(Q)}}\\
\lesssim
\| t^{-d} \int |1_{B(x,t)}-1_{B(y,t)}| (\sum_{i} |b_{i}|^{q})^{1/q}\|_{L^{r}_{t>\diam(Q)}}\\
\lesssim
\diam(Q) \| t^{-1} \|_{L^{r}_{t>\diam(Q)}}
\lesssim
\diam(Q)^{1/r}
\end{multline*}
for the first factor in \eqref{eq:BMO-local2}.
In the second factor we split
\[
|\tdif{}{t}(A_{t}b_{i}(x) - A_{t}b_{i}(y))|
\lesssim
t^{-d-1} |\int_{B(x,t) \Delta B(y,t)} b_{i}| + t^{-d} |\tdif{}{t}\int_{B(x,t)} b_{i}| + t^{-d} |\tdif{}{t} \int_{B(y,t)} b_{i}|
\]
and estimate the contributions of the first and the second summand separately (the third summand is entirely analogous to the second).
The contribution of the first summand is estimated as above and gives $\|\dots\|_{L^{r}}\lesssim\diam(Q)^{-1+1/r}$.
In the second summand note
\[
\tdif{}{t}\int_{B(x,t)} b_{i} \approx \int_{\partial B(x,t)} b_{i},
\]
and analogously to the above case we again obtain $\|\dots\|_{L^{r}}\lesssim\diam(Q)^{-1+1/r}$.
Combining these estimates we see that \eqref{eq:BMO-local2} is bounded by an absolute constant, and this concludes the proof.

\appendix
\section{A good \texorpdfstring{$\lambda$}{lambda} inequality}
The good $\lambda$ inequality below provides an alternative way to deduce part~\ref{thm:square:Ap} of Theorem~\ref{thm:square} from the special case $p=2$, $w\equiv 1$.
\begin{proposition}
Suppose $\s : L^{p} \to L^{p,\infty}$ (unweighted) for some $1\leq p<\infty$.
Let $w$ be a weight satisfying the $A_{\infty}$ condition \eqref{eq:Ainfty}.
Then for each $\lambda > 0$ and $A>1$ we have
\[
w\{ \s f > A \lambda, M^{\sharp}_{p} f \leq \gamma \lambda \}
\lesssim_{p}
C_{w} \big( \frac{\gamma}{A-1} \big)^{p\delta} w \{ \s f > \lambda \}.
\]
\end{proposition}
\begin{proof}
We decompose
\[
\{ \s f > \lambda \} = \bigcup_{\mathcal{Q}} Q
\]
into a disjoint union of maximal dyadic cubes $Q$.
We have
\begin{align*}
\frac{w \{\s f > A \lambda, M_{p}^{\sharp} f \leq \gamma \lambda \}}{w\{ \s f > \lambda \}}
&\leq
\sup_{Q\in\mathcal{Q}} \frac{w \{\s f > A \lambda, M_{p}^{\sharp} f \leq \gamma \lambda, Q \}}{w(Q)}\\
&\leq
C_{w} \sup_{Q\in\mathcal{Q}} \big(\frac{|Q\cap \{\s f > A \lambda, M_{p}^{\sharp} f \leq \gamma \lambda\} |}{|Q|}\big)^{\delta}
\end{align*}
by the $A_{\infty}$ condition.
Fix $Q\in\mathcal{Q}$, we need to find a bound for the ratio on the right-hand side.
Now, if $M^{\sharp}_{p} f  > \gamma \lambda$ on $Q$, then the ratio is zero, so we assume that there exists some $z = z_Q$ such that
\[
M_{p}^{\sharp}f (z) \leq \gamma \lambda.
\]
Suppose $Q\in\sigma_{j}$.
Now, for $x \in Q$, we may express
\[
(\s f)^2 (x)
\lesssim
\sum_{k \leq j} \sum_{Q' \sim \mathcal{Q}_{k}(x)} S_k f(x_{Q'})^2
+
\sum_{k > j} \sum_{Q' \sim \mathcal{Q}_{k}(x)} S_k f(x_{Q'})^2.
\]

Here $Q' \sim \mathcal{Q}_{k}(x)$ means the sum is taken over all neighbors $Q'$ of $\mathcal{Q}_{k}(x)$, the unique dyadic cube with side length $2^k$ containing $x$, and $x_{Q'}$ is a point at which the supremum over $Q$ is almost attained.

By maximality of $Q$, the second term is $\leq \lambda^{2}$.
On the other hand, the first term in the sum may be expressed as
\[
\sum_{k \leq j} \sum_{Q' \sim \mathcal{Q}_{k}(x)} S_k ((f-c)1_{CQ})(x_{Q'})^2
\leq
(\s ((f-c) 1_{CQ} ))^{2}
\]
for any $c$.
Thus
\[
Q\cap \{ \s f > A \lambda, M^{\sharp}_{p} f \leq \gamma \lambda \} \subset
Q\cap \{ \s ((f-c) \cdot 1_{CQ} ) > (A-1) \lambda \}.
\]
Using the weak-type $(p,p)$ boundedness of $\s$, we may estimate the measure of this final set by
\[
\frac{1}{(A-1)^{p} \lambda^{p}} \|\s\|_{L^p \to L^{p,\infty} }^{p} \int_{CQ} |f-c|^{p}.
\]
Minimizing over $c$ we obtain the estimate
\[
\frac{1}{(A-1)^{p} \lambda^{p}} |CQ| M^{\sharp}_{p}f(z)^{p}
\lesssim \frac{\gamma^{p}}{(A-1)^{p}} |Q|.
\qedhere
\]
\end{proof}

\section{A Sobolev embedding theorem for \texorpdfstring{$r$}{r}-variation}
For completeness we recall a simple estimate for the $r$-variational norm of a function on $\R$ in terms of $L^{r}$ norms.
\begin{lemma}
\label{lem:besov-vr}
Let $X$ be a normed space, $1\leq r<\infty$, and $a: [0,T] \to X$.
Then
\[
\hV{a}{} \leq 8 \|a\|_{L^{r}(X)}^{1-1/r} \|a'\|_{L^{r}(X)}^{1/r}.
\]
\end{lemma}
Note that the conclusion does not explicitly depend on $T$, so the estimate remains true for functions defined on infinite intervals.
\begin{proof}
Let $L>0$ be an integer chosen later and set $D:=\{0,\frac{T}{L},\dots,(L-1)\frac{T}{L}\}$.
Let also $\delta \in [0,\frac{T}{L}]$ be chosen later.
Splitting the variation into a short and a long part we get
\begin{align*}
\hV{a}{}
&\leq
\hV{a(t)}{t\in\delta+D} + 2 \big( \sum_{j=-1}^{L} \hV{a(t)}{t\in [0,T] \cap j\frac{T}{L}+\delta+[0,\frac{T}{L}]}^{r} \big)^{1/r}\\
&\leq
2 \| a(t) \|_{\ell^{r}_{t\in\delta+D}} + 2 \big( \sum_{j=-1}^{L} \hV[1]{a(t)}{t\in [0,T] \cap j\frac{T}{L}+\delta+[0,\frac{T}{L}]}^{r} \big)^{1/r}.
\end{align*}
We have $\|a\|_{L^{r}}^{r} = \int_{\delta=0}^{\frac{T}{L}} \| a(t) \|_{\ell^{r}_{t\in\delta+D}}^{r}$, so for some $\delta$ the first summand is bounded by $2 (T/L)^{-1/r} \|a\|_{L^{r}}$.
In the second summand we have estimated the $r$-variation norm by the $1$-variation norm, which can be estimated by the $L^{1}$ norm of the derivative.
Hence we obtain
\[
\hV{a}{}
\leq
2 (T/L)^{-1/r} \| a \|_{L^{r}} + 2 \big( \sum_{j=-1}^{L} \|a'\|_{L^{1}_{t\in [0,T] \cap j\frac{T}{L}+\delta+[0,\frac{T}{L}]}}^{r} \big)^{1/r}.
\]
We estimate the $L^{1}$ norm on the right-hand side by the $L^{r}$ norm using the Hölder inequality, this gives
\begin{align*}
\hV{a}{}
&\leq
2 (T/L)^{-1/r} \| a \|_{L^{r}} + 2 (T/L)^{1-1/r} \big( \sum_{j=-1}^{L} \|a'\|_{L^{r}_{t\in [0,T] \cap j\frac{T}{L}+\delta+[0,\frac{T}{L}]}}^{r} \big)^{1/r}\\
&=
2 (T/L)^{-1/r} \| a \|_{L^{r}} + 2 (T/L)^{1-1/r} \|a'\|_{L^{r}}.
\end{align*}
The optimal choice $L= \lfloor T \|a'\|_{L^{r}}/ \|a\|_{L^{r}} \rfloor$ gives the desired estimate unless $T \|a'\|_{L^{r}} < \|a\|_{L^{r}}$.
However, in the latter case we have the easy bound
\[
\hV{a}{}
\leq
\hV[1]{a}{}
\leq
\|a'\|_{L^{1}}
\leq
T^{1-1/r}\|a'\|_{L^{r}}
\leq
\|a\|_{L^{r}}^{1-1/r} \|a'\|_{L^{r}}^{1/r}.
\qedhere
\]
\end{proof}

\printbibliography
\end{document}
